\documentclass
{article}
\usepackage[T2A]{fontenc}
\usepackage[cp1251]{inputenc}
\usepackage[english,russian]{babel}
\usepackage[tbtags]{amsmath}
\usepackage{amsfonts,amssymb,mathrsfs,amscd}
\usepackage{wasysym}
\usepackage{verbatim}
\usepackage{eucal}
\usepackage{upgreek}
\usepackage{graphicx}
\usepackage{enumerate}
\let\No=\textnumero

\usepackage[hyper]{msb-a}
\JournalName{}
\numberwithin{equation}{section}

\theoremstyle{plain}

\newtheorem{maintheorem}{Основная теорема}
\newtheorem{propos}{Предложение}
\newtheorem{lemma}{Лемма}

\newtheorem{thN}{Теорема Р. Неванлинны}

\newtheorem{lemGS}{Лемма Гришина\,--\,Содина о малых интервалах}
\newtheorem{LemmaEF}{Лемма Эдрея\,--\,Фукса о малых дугах}
\newtheorem{ThSI}{Теорема о малых интервалах}
\newtheorem{ThSIw}{Теорема о малых интервалах с весом}
\newtheorem{ThSp}{Теорема о малых плоских подмножествах}
\newtheorem{thGM}{Теорема Гришина\,--\,Малютиной о малых интервалах}

\theoremstyle{definition}
\newtheorem{proof}{Доказательство}
\newtheorem{remark}{Замечание}
\newtheorem{definition}{Определение}

\renewcommand{\leq}{\leqslant} 
\renewcommand{\geq}{\geqslant}
\newcommand{\RR}{\mathbb{R}} 
\newcommand{\CC}{\mathbb{C}} 
\newcommand{\NN}{\mathbb{N}}

\DeclareMathOperator{\dd}{\,{\mathrm d\!}}

\DeclareMathOperator{\rad}{{\text{\tiny \rm rd}}}

\DeclareMathOperator{\mes}{mes}

\begin{document} 
\title{Интегралы  от разностей субгармонических  функций.~I. Интегральное неравенство с  характеристикой Неванлинны и модулем непрерывности  меры}
	
\author[B.\,N.~Khabibullin]{Б.\,Н.~Хабибуллин}
\address{Башкирский государственный университет}
\email{khabib-bulat@mail.ru}

\date{25.05.2021}
\udk{517.547.2 : 517.574}

 \maketitle

\begin{fulltext}

\begin{abstract} Получены новые  интегральные неравенства для интегралов от разности субгармонических функций
по мере  через их характеристику Неванлинны и некоторую функциональную характеристику меры --- её модуль непрерывности.  Эти результаты новые и для мероморфных функций.     

Библиография:  32 названия 

Ключевые слова: мероморфная функция, разность субгармонических функций, характеристика Неванлинны,  мера Рисса,  модуль непрерывности, мера Хаусдорфа


\end{abstract}

\markright{Интегралы  от разностей субгармонических  функций \dots}

	
\section{Введение}

\subsection{Цель и истоки}\label{s10}
Основные задачи всей работы и этой первой её части  --- верхние неравенства для  интегралов от разности  субгармонических функций по мерам на подмножествах  {\it комплексной плоскости\/}  $\CC$ или конечномерного  евклидова пространства  через их характеристику Неванлинны и количественные параметры меры и/или множества интегрирования. Их следствия для $\CC$ ---  новые неравенства   для интегралов   от логарифма модуля мероморфной функции по мерам. Исходная  точка исследования  --- 
лемма А.~Эдрея и В.~Фукса о малых дугах, которая нашла важные  применения в теории мероморфных функций, отражённые, например, в  
\cite[гл.~I, теорем  7.4]{GO}. Приводится она в формулировке из \cite{GO} с единственным отличием: здесь и 
далее через  $\uplambda_{\RR}$  обозначаем {\it линейную  меру Лебега\/} на $\RR$ в отличие от широко распространённого обозначения $\mes$ в \cite{GO}, \cite{GrS}, \cite{GriMal05}, \cite{GabKha20}, \cite{Kha20} и т.д.
Для мероморфной функции $f\neq 0,\infty$ на $\CC$ 
 её  {\it характеристика Неванлинны\/}  на {\it положительном луче\/} $\RR^+:=\bigl\{x\in \RR\bigm| x\geq 0\bigr\}$ 
{\it вещественной прямой\/} $\RR$ определяется  как 
\begin{subequations}\label{TN}
	\begin{align}
	T(r, f)&\underset{r\in\RR^+}{:=}m(r,f)+N(r,f),  \quad\text{где}
	\tag{\ref{TN}T}\label{{TN}T}
	\\
	m(r,f)&\underset{r\in\RR^+}{:=}\frac{1}{2\pi}\int_0^{2\pi} \ln^+\bigl|f(re^{i\varphi})\bigr| \dd \varphi , 
\quad \ln^+x\underset{x\in \RR^+}{:=}\max\{0, \ln x\},
\tag{\ref{TN}m}\label{{TN}m}\\
N(r,f)&\underset{r\in\RR^+}{:=}\int_{0}^{r}\frac{n(t,f)-n(0,f)}{t}\dd t+n(0,f)\ln r, 
	\tag{\ref{TN}N}\label{{TN}N}
	\end{align}
\end{subequations} 
а   $n(r,f)$ --- число полюсов функции  $f$ в  замкнутом круге радиуса $r\in \RR^+$ с центром в нуле, подсчитанное с учётом кратности полюсов. 

\begin{LemmaEF}[{\cite[{\bf 2}, лемма III, {\bf 9}]{EF}, \cite[теорема 7.3]{GO}}]
Пусть $f$ --- мероморфная функция, $k$ и $\delta$ --- некоторые числа, $k>1$, $0< \delta \leq 2\pi$, $r>1$. 
Существует такая постоянная $c_1(k,\delta)$, что для любого $\uplambda_{\RR}$-измеримого $E_r\subset [-\pi,\pi]$ такого, что $\uplambda_{\RR}(E_r)=\delta$, выполняется 
\begin{equation}\label{eEF}
\int_{E_{r}} \ln^+\bigl|f(re^{i\varphi})\bigr|\dd \varphi\leq c_1(k,\delta) T(kr,f),
\end{equation}    
где $c_1(k,\delta)=\frac{6k}{k-1}\delta \ln \frac{2\pi e}{\delta}\to 0$, когда $\delta \to 0$ при фиксированном $k$.  
\end{LemmaEF}
Перенос этой леммы на разности  субгармонических функций и снятие  ограничения $r>1$ дополнением правой части \eqref{eEF} слагаемым порядка $O(\ln r)$ при таком же интегрировании 
обсуждается у М. Гирныка  \cite[теорема E и ниже]{Hir09}. 

В ряде результатов для мероморфных функций $f$ на $\CC$, когда множества интегрирования  выбирались  на {\it положительной  полуоси\/} $\RR^+$, имелась возможность устанавливать  оценки 
для интегралов от положительной  части логарифма {\it максимума модуля\/} мероморфной функции 
\begin{equation}\label{{TN}M}
M(r,f)\underset{r\in \RR^+}{:=}\sup\Bigl\{\bigl|f(z)\bigr| \Bigm| |z|=r\Bigr\}. 
\end{equation}
Если учитывать такие результаты, то  лемме Эдрея\,--\,Фукса предшествовала 
\begin{thN}[{\rm  (\cite[гл. I, теорема 7.2]{GO} вместе с   обсуждением  в   \cite[Введение, 1.1]{Kha20})}] 
Пусть   $1<k\in \RR^+$, $0<r_0\in \RR^+$. Тогда существует такое число $c_0(k)\in \RR^+$, что для любой 
 мероморфной функции $f\neq 0,\infty$   на\/ $\CC$ 
\begin{equation}\label{eN}
\int_0^r\ln^+M(t,f)\dd t\leq c_0(k) T(kr,f)r \quad\text{при всех $r\geq r_0$.} 
	\end{equation}
\end{thN}
<<Прямолинейная>> версия  леммы Эдрея\,--\,Фукса о малых дугах --- 
\begin{lemGS}[{\rm \cite[лемма 3.1]{GrS}}] Существует $C\in \RR^+$, с которым  для  любых 
мероморфной   функции   $f\neq 0,\infty$, чисел  $r>1$ и $k>1$, а также  $\uplambda_{\RR}$-измеримого     $E\subset [1,r)$
выполнено неравенство 
	\begin{equation}\label{eGS}
\int_E\ln^+M(t,f)\dd t\leq 	C\frac{k}{k-1}T(kr,f)\uplambda_{\RR}( E)\ln 	
\frac{2r}{\uplambda_{\RR}( E)}.
	\end{equation}
\end{lemGS}
В совместной работе А.\,Ф.~Гришина и Т.\,И.~Малютиной  \cite{GriMal05}  
доказана \cite[теорема 8]{GriMal05} и неоднократно применяется \cite[теоремы 2, 4]{GriMal05}
версия леммы Гришина\,--\,Содина {\it для субгармонических функций формального уточнённого  порядка
 $\boldsymbol \rho$ в смысле Валирона\/} \cite[гл.~III, 6]{Valiron}, \cite[гл.~I, \S~12]{Levin56}, \cite[гл.~II, \S~2]{GO}, \cite[7.4]{BGT}, который можно  определить в эквивалентной форме  через {\it единственное условие\/}   как 
{\it  дифференцируемую функция $\boldsymbol \rho\geq 0$ на  $\RR^+\setminus \{0\}$} 
с  {\it  конечным пределом\/} $\rho :=\lim\limits_{r\to +\infty}r\bigl(\boldsymbol \rho(r)\ln r\bigr)'\in \RR^+$, что отмечено  в  
 \cite[следствие]{Kha20D}  и  для существенно  более общих {\it уточнённых функций роста}
 \cite[теорема]{Kha20D}. 
\begin{thGM}[{\rm \cite[теорема 8]{GriMal05}}] 
Пусть $v\not\equiv -\infty$ --- субгармоническая функция на $\CC$ и
$\sup\limits_{z\in \CC}v(z)|z|^{-\boldsymbol \rho (|z|)}<+\infty$ для уточнённого порядка $\boldsymbol \rho $.
Тогда существует $M\in \RR^+$, с которым для любых числа $r>1$ и $\uplambda_{\RR}$-измеримого множества
$E\subset [1,r)$ выполняется неравенство
\begin{equation}\label{estGM}
\sup_{\theta\in [0,2\pi)}\int_E\bigl|v(te^{i\theta})\bigr| \dd t\leq 
M r^{\boldsymbol \rho (r)}\uplambda_{\RR}( E) \ln \frac{4r}{\uplambda_{\RR}(E)}.
\end{equation}
 \end{thGM}

Теорема Гришина\,--\,Малютиной о малых интервалах недавно  была распространена на произвольные субгармонические функции на $\CC$ в совместной статье Л.\,А.~Габдрахмановой и автора   \cite{GabKha20},  
из результатов которой она легко следует \cite[вывод теоремы Гришина\,--\,Малютиной]{GabKha20}.   

Для {\it расширенной числовой функции\/} $v$ на окружности радиуса $r\in \RR^+$ с центром в нуле   со значениями из {\it расширенной числовой прямой\/} $\overline \RR:=\RR\cup\{\pm\infty\}$ 
\begin{subequations}\label{MC}
\begin{align}
{\sf M}_v(r)&:=\sup_{0\leq \theta < 2\pi } v(re^{i\theta}),    
\tag{\ref{MC}M}\label{Mr}\\
{\sf C}_v(r)&:=\frac{1}{2\pi}\int_0^{2\pi} v(re^{i\theta})\dd \theta
\tag{\ref{MC}C}\label{Cr}
\end{align}
\end{subequations} 
 --- {\it среднее по\/} этой {\it окружности\/} при условии интегрируемости $v$.  
\begin{ThSI}[{\cite[теорема 1]{GabKha20}}]\label{thGK} 
Существует число $a\geq 1$, с которым  для любых  субгармонической на $\CC$ функции $u\not\equiv -\infty$, 
чисел\/  $b\in (0,1]$  и $0\leq r_0\leq  r<R<+\infty$, $\uplambda_{\RR}$-измеримого подмножества $E\subset [r,R]$,
а также  функции $g\colon E\to \overline \RR$ с  существенной верхней гранью 
\begin{equation}\label{Lin}
\|g\|_\infty:=\inf \Bigl\{a\in \RR\Bigm| \uplambda_{\RR}\Bigl(\bigl\{x\in E\bigm|
g(x)>a\bigr\}\Bigr)=0 \Bigr\}\in \RR^+
\text{ на $E$}
\end{equation}
имеет место неравенство 
\begin{equation}\label{M|u|}
\int_{E}{\sf M}_{|u|} g\dd \uplambda_{\RR}
\leq \left(\frac{a}{b}\ln \frac{a}{b}\right)
 \Bigl( {\sf M}_u\bigl((1+b)R\bigr)+2{\sf C}_u^-(r_0)\Bigr) \|g\|_\infty \uplambda_{\RR}( E) \ln \frac{3beR}{ \uplambda_{\RR}( E)}.
\end{equation}
\end{ThSI}

Ещё один общий результат  установлен в  \cite[основная теорема]{Kha20}.
 Для    меры  $\mu$  на замкнутом круге радиуса $R$ с центром в нуле 
  через $\mu^{\rad}(r)$ обозначим $\mu$-меру замкнутых кругов радиуса $r\leq R$ с центром в нуле, а   
\begin{equation}\label{{murad}N}
{\sf N}_{\mu}(r,R):=\int_{r}^{R}\frac{\mu^{\rad}(t)}{t}\dd t\in \overline \RR^+:=\RR^+\cup \{+\infty\}, \quad 
0\leq r<R\leq +\infty .
\end{equation} 

Для числа или  функции 
верхний индекс $+$ определяет его положительную часть, а знак $-$ в верхнем индексе выделяет  отрицательную   часть. 

Пусть  $U=u-v$ --- разность субгармонических функций 
 $u\not\equiv -\infty$ и $ v\not\equiv -\infty$ в окрестности замкнутого круга радиуса $R$ с мерами Рисса соотвественно $\varDelta_u\geq 0$ и $\varDelta_v\geq 0$, т.\,е.  {\it $\delta$-суб\-га\-р\-м\-о\-н\-и\-ч\-е\-с\-к\-ая нетривиальная\/}  ($\not\equiv\pm\infty$) {\it функция\/}  \cite{Arsove53}, \cite{Arsove53p}, \cite{Gr}, \cite[3.1]{KhaRoz18}
с зарядом Рисса $\varDelta_{U}=\varDelta_u-\varDelta_v$. 
В наших статьях \cite{Kha20} и \cite{Kha21} использовалась 
{\it разностная  характеристика Неванлинны\/} функции $U$ вида  
\begin{equation}\label{T}
{\sf T}_U(r,R)={\sf C}_{U^+}(R)-{\sf C}_{U^+}(r)+
{\sf N}_{\varDelta_U^-}(r,R)
, \quad 0<r< R\in \RR^+, 
\end{equation}
где  мера Бореля $\varDelta_U^-
\geq 0
$ --- нижняя вариация заряда Рисса  $\varDelta_{U}$. 

Для $\lambda_{\RR}$-измеримых  $E\subset \RR$ и   $g\colon E\to \overline \RR$ наряду с существенной верхней гранью  $\|g\|_{\infty}$ на $E$ из 
\eqref{Lin} используем и $L^p$-полунорму  функции $g$ на $E$:  
\begin{equation}\label{Lpn}
 \|g\|_p:=\sqrt[p]{\int_E |g|^p\dd\uplambda_{\RR}}\quad \text{при  $1\leq p\in \RR^+$ на $E$}.
\end{equation}

\begin{ThSIw}[{\cite[основная теорема]{Kha20}}]\label{thm} 
Пусть   $0< r_0< r<+\infty$, $1<k\in \RR^+$, $\uplambda_{\RR}$-измеримы $E\subset [0,r]$ и $g\colon E\to \overline \RR$, 
 $1<p\leq \infty$,  $1/p+1/q=1$,  $U\not\equiv \pm\infty$  --- $\delta$-субгармоническая функция на  $\CC$, 
а  $u\not\equiv -\infty$ --- субгармоническая функция на $\CC$.  Тогда  
\begin{subequations}\label{1m}
\begin{align}
\int_{E} {\sf M}_{U}^+(t)g(t)\!\dd t&\leq
\frac{4qk}{k-1} \bigl({\sf T}_{U}(r_0,kr)+{\sf C}_{U^+}(r_0)\bigr) \|g\|_p
\sqrt[q]{\uplambda_{\RR}( E)} \ln\frac{4kr}{\uplambda_{\RR}( E)}, \tag{\ref{1m}T}\label{inDl+}
\\ 
\int_{E} {\sf M}_{|u|}(t)g(t)\!\dd t&\leq
\frac{5qk}{k-1} \bigl({\sf M}_{u^+}(kr)+{\sf C}_{u^-}(r_0)\bigr) \|g\|_p
\sqrt[q]{\uplambda_{\RR} (E)}\ln\frac{4kr}{\uplambda_{\RR} (E)}.
\tag{\ref{1m}M}\label{uM}
\end{align}
\end{subequations}
\end{ThSIw}

Следующий  в определённом смысле  качественно новый шаг позволяет охватить  
и  множества малой {\it плоской меры Лебега\/} $\uplambda_{\CC}$ на  комплексной плоскости.
\begin{ThSp}[{\cite[теорема 2, следствия]{Kha21}}]
Пусть   $0< r_0< r\in \RR^+$, $1<k\in \RR^+$,  $E$ --- $\uplambda_{\CC}$-измеримое подмножество в замкнутом круге
радиуса $r$ с центром в нуле,  $U\not\equiv \pm\infty$ --- $\delta$-субгармоническая  функция на $\CC$, а $u\not\equiv -\infty$ --- субгармоническая  функция на $\CC$. Тогда  
\begin{subequations}\label{1pl}
\begin{flalign}
\int_{E} U^+\dd \uplambda_{\CC} &\leq
\frac{2k}{k-1} \Bigl({\sf T}_{U}(r_0,kr)+{\sf C}_{U^+}(r_0)\Bigr)\uplambda_{\CC}(E)
\ln\frac{100kr^2}{\uplambda_{\CC} (E)},
\tag{\ref{1pl}T}\label{inDl+p}
\\ 
\int_{E} |u|\dd \uplambda_{\CC} &\leq
\frac{3k}{k-1} \Bigl({\sf M}_{u^+}(kr)+{\sf C}_{u^-}(r_0)\Bigr) \uplambda_{\CC}(E) \ln\frac{100kr^2}{\uplambda_{\CC} (E)}.
\tag{\ref{1pl}M}\label{uMp}
\end{flalign}
\end{subequations}
\end{ThSp}

\subsection{О содержании работы}
Для мероморфной функции $f\neq 0,\infty$ на $\CC$ её логарифм модуля $\ln |f|\not\equiv \pm \infty$ ---  $\delta$-субгармоническая функция на  ${\mathbb{C}}$ и все установленные ниже результаты  о $\delta$-субгармонических функциях новые и для мероморфных функций в  традиционных  обозначениях 
\begin{subequations}\label{cs}
\begin{align}
\ln M(r, f)&\overset{\eqref{{TN}M},\eqref{Mr}}{=}{\sf M}_{\ln|f|}(r), \quad r\in {\mathbb{R}}^+,
\tag{\ref{cs}M}\label{{cs}M}\\  
m(r, f)&\overset{\eqref{{TN}m},\eqref{Cr}}{=}{\sf C}_{\ln^+|f|}(r),\quad r\in {\mathbb{R}}^+,
\tag{\ref{cs}m}\label{{cs}m}\\
N(R, f)-N(r, f)&\overset{\eqref{{TN}N},\eqref{{murad}N}}{=}
{\sf N}_{\varDelta_{\ln|f|}^-}(r,R),
\quad 0<r< R\in \RR^+, 
\tag{\ref{cs}N}\label{{cs}N}\\
T(R, f)-T(r, f)&\overset{\eqref{{TN}T},\eqref{T}}{=}{\sf T}_{\ln|f|}(r,R).
\quad 0<r<R\in {\mathbb{R}}^+,
\tag{\ref{cs}T}\label{{cs}T}
\end{align}
\end{subequations}

В связи с видом круглой скобки в правых  частях \eqref{inDl+} и \eqref{inDl+p}  с двумя слагаемыми далее будет удобнее  использовать именно эту сумму как  {\it разностную характеристику Невалинны,\/} которую  можно определить через предшествующую форму  разностной характеристики Невалинны ${\sf T}_{U}$ из \eqref{T} в виде 
\begin{subequations}\label{rT}
\begin{align}
{\boldsymbol T}_U(r,R)&:={\sf T}_{U}(r,R)+{\sf C}_{U^+}(r)
\tag{\ref{rT}T}\label{{rT}T}
\\
&\overset{\eqref{T}}{=}{\sf C}_{U^+}(R)+{\sf N}_{\varDelta_U^-}(r,R),
\quad 0<r<R\in \RR^+,
\tag{\ref{rT}N}\label{{rT}N}
\\
\intertext{где правая часть теперь позволяет определить и}
{\boldsymbol  T}_U(R)&:={\boldsymbol  T}_U(0,R)\overset{\eqref{{murad}N}}{:=}{\sf C}_{U^+}(R)+{\sf N}_{\varDelta_U^-}(0,R)\in \overline \RR^+.
\tag{\ref{rT}o}\label{{rT}o}
\end{align}
\end{subequations}
В этом случае \eqref{{cs}T} согласно \eqref{{cs}m}, \eqref{{rT}N} и \eqref{{cs}N}  заменится  на 
\begin{equation}\label{csTTNN}
T(R, f)-N(r, f)={\boldsymbol  T}_{\ln|f|}(r,R),
\quad 0< r<R\in {\mathbb{R}}^+.
\end{equation}

Все итоговые оценки сверху интегралов в статье будут даваться через последний вариант  ${\boldsymbol  T}_U$ разностной характеристики Неванлинны из \eqref{rT} только для $\delta$-суб\-г\-а\-р\-м\-о\-н\-и\-ч\-е\-с\-к\-их функций, поскольку взаимосвязи  \eqref{cs}, а прежде всего  \eqref{csTTNN}, очевидным образом позволяют переформулировать их для мероморфных функций в традиционных обозначениях \eqref{TN} и \eqref{{TN}M}.
 
Оценки интегралов по подмножествам положительной полуоси от логарифмов максимума модуля \eqref{{TN}M} мероморфной функции $f$, как в теореме Р.~Неванлинны и лемме Гришина\,--\,Содина,  от функций ${\sf M}_{|u|}$ с субгармонической функцией $u$, как  в \eqref{M|u|} и \eqref{uM},  а также от  ${\sf M}_{U^+}$, как  в \eqref{inDl+}, 
в теоремах о малых интервалах   --- это  отдельная  специфическая  задача, которая тесно завязана на интегралах именно по интервалам и подмножествам $E$ в $\RR^+$, но   включает в себя суперпозицию двух  операций: точная верхняя грань по окружностям из   \eqref{{TN}M} или \eqref{Mr} с последующим интегрированием. 
Она рассмотрена  нами отдельно в \cite{Kha21A1}--\cite{Kha21A3} и намечена к объединяющей публикации в открытой печати в одном из российских журналов.
В настоящей работе  устанавливаются неравенства для интегралов от собственно  функций $\ln^+ |f|$ с мероморфной $f$ на $\CC$ или $U^+$, где $U$ --- разность субгармонических  функций, по мерам на подмножествах произвольной локации соответственно в $\CC$ или в конечномерном евклидовом пространстве размерности $\geq 2$. Это  можно трактовать как развитие и обобщение именно леммы Эдрея\,--\,Фукса о малых дугах, в какой-то мере теоремы Гришина\,--\,Малютиной о малых интервалах, а также теоремы о малых плоских подмножествах в $\CC$. 

\section{Разностная характеристика Неванлинны}

{\it Расширенная  числовая  прямая\/} $\overline \RR:=\RR\cup \{\pm \infty\}$ ---  двухточечная компактификация  $\RR$ путём  добавления двух концов $\inf \RR=:-\infty=:\sup \varnothing $ и  $\sup \RR=:+\infty=:\inf \varnothing$, где $\varnothing$ --- {\it  пустое множество}, дополненная отношениями порядка  
$-\infty \leq x\leq +\infty$ для всех $x\in \overline \RR$, и операциями
\begin{subequations}\label{actR}
\begin{align}
 -(\pm\infty)&=\mp\infty, \quad |\pm\infty|:=+\infty,
\tag{\ref{actR}$\infty$}\label{{actR}-}
 \\   
x\pm(\pm \infty)&=+ \infty \text{ при $x\in \overline \RR\!\setminus\!-\infty$},
\; x\pm (\mp \infty)=-\infty\text{ при }x\in \overline \RR\!\setminus\!+\infty,
\tag{\ref{actR}$\pm$}\label{{actR}+}
\\
x\cdot (\pm\infty)&:=\pm\infty=:(-x)\cdot (\mp\infty) \text{ при $x\in \overline \RR^+\!\setminus\!0$}, 
\tag{\ref{actR}*
}\label{{actR}.}
\\
\frac{\pm x}{0}&:=\pm\infty\text{ при  $x\in  \overline \RR^+\!\setminus\!0$}, \; 
\frac{x}{\pm\infty}:=0\text{ при  $x\in  \RR$}, 
\tag{\ref{actR}/}\label{{actR}/}\\
\text{но } 0&\cdot\pm \infty:=0=:\pm \infty \cdot 0, \text{ если не оговорено иное,}
\tag{\ref{actR}$\cdot0$}\label{{actR}0}\\
\intertext{а не определены только  пары сумм и разностей и пять операций деления} 
\nexists!\bigl( (\pm\infty)&+(\mp\infty)\bigr),  \quad \nexists!\bigl((\pm\infty)-(\pm \infty)\bigr),
\quad \nexists!\frac{0}{0},  \quad \nexists!\,\frac{\pm\infty}{\pm\infty}, \quad \nexists!\,\frac{\pm\infty}{\mp\infty},
\tag{\ref{actR}$\nexists!$}\label{{actR}ne}
\end{align}
\end{subequations}   

Одноточечные множества записываем без фигурных скобок, если это не вызывает разночтений. 

{\it Интервал\/} на $\overline \RR$ --- связное подмножество в $\overline \RR$. 
 При $a\leq b\in \overline \RR$, как обычно,  
$[a,b]:=\bigl\{x\in \overline \RR\bigm| a\leq x\leq b\bigr\}$    --- {\it отрезок\/} на  $\overline \RR$ с {\it левым концом\/} $a$ и {\it правым концом\/} $b$, $(a,b]:=[a,b]\setminus a$ --- {\it открытый слева и замкнутый справа интервал\/} на  $\overline \RR$, аналогично для  $[a,b):=[a,b]\setminus b$, $(a,b):=(a,b]\cap [a,b)$ --- {\it открытый интервал.\/}
{\it Интеграл\/} (Римана\,--\, или  Лебега\,--\,){\it Стилтьеса\/}
по интервалу с концами $a<b$  по функции ограниченной вариации $g$ на этом интервале
понимаем как интеграл по   $(a,b]\subset \overline \RR$, если не оговорено иное:
\begin{equation}\label{keyint}
\int_a^b \dots \dd g:=\int_{(a,b]} \dots \dd g.
\end{equation}

Через $x^+:=\sup \{0,x\}$ обозначаем {\it положительную часть\/} от 
 $x\in \overline \RR$, а $x^-:=(-x)^+$ --- его {\it отрицательная часть\/}.    Вообще всюду далее
{\it положительность\/} --- это $\geq 0$, а {\it отрицательность\/} --- это $\leq 0$. Если $0<x\in \overline \RR$, то $x$ {\it строго\/} положительно, а если $0>x\in \overline \RR$, то $x$ {\it строго\/} отрицательно.  
Для {\it расширенной числовой функции\/} $f\colon X\to \overline \RR$, вообще говоря,
могут быть  и не определены значения $f(x)$ для некоторых $x$.  При этом её её {\it положительная  часть\/}
 $f^+\colon x\underset{\text{\tiny $x\in X$}}{\longmapsto} \bigl(f(x)\bigr)^+$ определена в тех же точках, что и $f$, а $f^{-}:=(-f)^+$ --- её {\it отрицательная  часть.\/} Функция $f$  {\it положительна на $X$\/} и пишем $f\geq 0$, если  $f=f^+$.

Всюду далее ${\tt d}\in \NN:=\{1,2,\dots\}$, ${\tt d}\geq 2$,  --- размерность {\it евклидова пространства\/} 
$\RR^{\tt d}$ с {\it евклидовой нормой\/} $|x|:=\sqrt{x_1^2+\dots +x_{\tt d}^2}$   
для $x:=(x_1,\dots ,x_{\tt d})\in \RR^{\tt d}$, а 
$B_x(r)$, $\overline B_x(r)$ и $\partial \overline B_x(r)$ ---  соответственно {\it открытый\/} и  {\it замкнутый шар,\/} а также  {\it  сфера радиуса $r\in \RR^+$
 с центром\/} $x\in \RR^{\tt d}$. По этому определению если $r=0$, то  $B_x(r)=\varnothing$ --- {\it пустое множество.\/}  Если рассматриваются шары    или сферы   с центром в нуле, то нижний индекс $0$, как правило,  не пишем:
\begin{equation*}
B(r):=B_0(r), \quad \overline B(r):= \overline B_0(r), \quad \partial \overline B(r)=\partial \overline B_0(r).
\end{equation*} 
Пространство $\RR^{\tt 2}$ часто  отождествляем с комплексной плоскостью 
$$
\CC\ni z=x+iy \longleftrightarrow (x,y)\in \RR^{\tt 2}, \quad x,y\in \RR,
$$ 
если удаётся отвлечься от комплексной структуры $\CC$. Для комплексной плоскости $\CC$, таким образом,  $B_z(r)$, $\overline  B_z(r)$ и  $\partial \overline B_z(r)$ --- соответственно {\it открытый\/} и  {\it замкнутый круг,\/} а также  {\it  окружность радиуса $r$ с центром\/} $z\in \CC$. 

{\it Площади поверхностей единичных сфер\/} $\partial \overline  B(1)$ в $\RR^{\tt d}$ обозначаем через 
\begin{equation}\label{{kK}s}
s_{\tt d-1}=\frac{2\pi^{\tt d/2}}{\Gamma (\tt d/2)},  
\quad s_{\tt 1}=2\pi, \; s_{\tt 2}=4\pi,\;  s_{\tt 3}=\pi^2, \dots  . 
\end{equation}
Для  {\it поверхностной меры площади\/} $\sigma_{\tt d-1}^r$ на сфере  $ \partial \overline B(r)\subset \RR^{\tt d}$ и интегрируемой по  $\sigma_{\tt d-1}^r$ функции $U\colon \partial \overline B(r)\to \overline \RR$  её {\it среднее по сфере $ \partial \overline B(r)$} обозначаем  как \eqref{Cr}
\begin{equation}\label{CU}
{\sf C}_U(r):=\frac{1}{s_{\tt d-1}r^{\tt d-1}}\int_{\partial \overline B(r)}U\dd \sigma_{\tt d-1}^r.
\end{equation}
Пусть   $\mu$ --- мера Бореля на $\RR^{\tt d}$ и 
\begin{equation}\label{muyr} 
\mu_y^{\rad}(t)\underset{t\in \RR^+}{:=}\mu\bigl(\overline B_y(t) \bigr)\in \overline \RR^+
\end{equation}
--- {\it радиальная считающая функция меры $\mu$ с центром $y\in \RR^{\tt d}$.\/}
В случае центра $y=0$ нижний индекс $0$, как правило, не используем, как и  перед \eqref{{murad}N} для $\CC$. 

Неоднократно будет использоваться связанное с размерностью ${\tt d}\in \NN$ число 
\begin{equation}\label{kd0}
\widehat{\tt d}:=\max\{{\tt 1,d-2}\}=1+({\tt d-3})^+\in \NN.  
\end{equation}
 Для меры Бореля $\mu$  на шаре $\overline B(R)\subset \RR^{\tt d}$,
как и в \eqref{{murad}N}  для $\CC$,  
\begin{equation}\label{Ntt}
{\sf N}_{\mu}(r,R):={\widehat{\tt d}}\int_r^R \frac{\mu^{\rad}(t)}{t^{\tt d-1}}\dd t\in \overline \RR^+ \quad\text{при  $ 0\leq r<R\in \overline \RR^+$}
\end{equation}
--- {\it  радиальная разностная проинтегрированная считающая функция меры $\mu$}.

При ${\tt d=1}$ субгармонические функции --- это выпуклые функции. 
По субгармоническим функциям при ${\tt d=2}$, т.е. для $\CC$, вполне достаточно сведений из  \cite{Rans}, а для остальных  ${\tt d}>2$ ---  из \cite{HK}, \cite{Helms}, \cite{Landkof}. 

Для субгармонической функции $u\not\equiv -\infty$ в открытом шаре $B(R)$ 
через  
\begin{equation}\label{df:cm}
\varDelta_u\overset{\eqref{kd0}}{:=} \frac{1}{s_{\tt d-1}{\widehat{\tt d}}} {\bigtriangleup}  u,
\quad\text{где ${\bigtriangleup}$ --- {\it оператор Лапласа,}} 
\end{equation}
действующий в смысле теории обобщённых функций на $B(R)$, обозначаем её распределение масс, или {\it меру Рисса.}
Функция $u\colon \overline B(r)\to \RR\cup -\infty$ {\it субгармоническая на замкнутом шаре\/} $\overline B(r)\subset \RR^{\tt d}$ радиуса $r>0$, если   функция $u$ равна  сужению на $\overline B(r)$ какой-нибудь  субгармонической функции в открытом шаре $B(R)\subset \RR^{\tt d}$ некоторого  строго большего радиуса $R>r$.   
Для {\it субгармонической функции\/} $u\not\equiv -\infty$ на $\overline B(r)\subset \RR^{\tt d}$ через указанное сужение  с $B(R)$ на $\overline B(r)$ вполне корректно определена её  {\it мера Рисса\/}  $\varDelta_u$ на замкнутом шаре  $\overline B(r)$, полученная сужением с $B(R)$ меры Рисса из  \eqref{df:cm}.

Формально функция $U\colon \overline B(r) \to \overline \RR$   {\it $\delta$-субгармоническая на $\overline B(r)\subset \RR^{\tt d}$,} если она задана как разность $U=u-v$ пары  субгармонических функций $u$ и $v$ на $\overline B(r)$.
Различные  эквивалентные формы определения  функций, их корректность и основные свойства  исследуются в \cite{Arsove53}, \cite{Arsove53p}, \cite[2.8.2]{Azarin}, \cite{Gr}, \cite[3.1]{KhaRoz18}.

Две $\delta$-субгармонические функции $U=u-v\not\equiv \pm \infty$ и $U_1=u_1-v_1\not\equiv \pm \infty$, представленные разностями пар субгармонических функций $u,v,u_1,v_1\not\equiv -\infty$ на открытом множестве, {\it равны\/} на этом множестве, если $u+v_1=u_1+v$ на нём.

Такая функция {\it нетривиальна\/} и пишем $U\not\equiv \pm\infty$, если в её представлении 
$U=u-v$ как $u\not\equiv -\infty$, так и $v\not\equiv -\infty$. Для  $\delta$-субгармонической
функции $U\not\equiv \pm\infty$  на $\overline B(r)\subset \RR^{\tt d}$ корректно определён
 её {\it заряд Рисса\/} $\varDelta_U:=\varDelta_u-\varDelta_v$ как разность мер Рисса субгармонических  функций $u\not\equiv
-\infty$ и $v\not\equiv -\infty$.

\begin{definition}\label{defT}
{\it Разностной характеристикой Неванлинны\/} ${\boldsymbol T}_U(r,R)$ для 
$\delta$-су\-б\-г\-а\-р\-м\-о\-н\-и\-ч\-е\-с\-к\-ой функции  $U\not\equiv \pm\infty$  на  шаре $\overline B(R) \subset \RR^{\tt d}$ радиуса $R>0$
называем сумму  из \eqref{{rT}N}  в обозначениях и определениях \eqref{CU} и  \eqref{Ntt}:
\begin{equation}\label{{TTN}T}
{\boldsymbol T}_U(r,R)\overset{\eqref{{rT}N}}{:=}{\sf C}_{U^+}(R)+{\sf N}_{\varDelta_U^-}(r,R)\in \overline \RR^+,
\quad 0\leq r<R\in  \RR^+,
\end{equation} 
\end{definition}
Определение разностной характеристики Неванлинны ${\boldsymbol  T}_U$ можно дать и  иначе.  Для $\delta$-суб\-г\-а\-р\-м\-о\-н\-и\-ч\-е\-с\-к\-ой функции $U\not\equiv \pm\infty$ на шаре  $\overline B(R)$ с зарядом Рисса $\varDelta_U$ существуют {\it канонические   представления\/} $U=u_*-v_*$, где $u_*\not\equiv -\infty$  и $v_*\not\equiv -\infty$  --- субгармонические функции на $\overline B(R)$ с мерами Рисса соответственно 
$\varDelta_{u_*}=\varDelta_U^+:=\sup\{0,\varDelta_U\}$ --- {\it верхняя вариация\/} 
заряда Рисса $\varDelta_U$ и $\varDelta_{v_*}=\varDelta_U^-$. Канонические представления  определены с точностью до общего гармонического слагаемого. Из  равенства $U^+=\sup\{u_*-v_*, 0\}=\sup\{u_*,v_*\}-v_*$ имеем
\begin{equation*}
{\sf C}_{U^+}(R)={\sf C}_{\sup\{u_*,v_*\}}(R)-{\sf C}_{v_*}(R)
\quad\text{для всех  $0<R<+\infty$},
\end{equation*}
где функция $\sup\{u_*,v_*\}$ субгармоническая на $\overline B(R)$, 
а по  формуле Пуассона\,--\,Йенсена\,--\,При\-в\-а\-л\-о\-ва 
 \cite{Pri35}--\cite{Pri35II}, \cite[гл. II, \S~2]{Privalov}, \cite[3.7]{HK}
\begin{equation*}
{\sf N}_{\varDelta_{v_*}}(r,R)={\sf C}_{v_*}(R)-{\sf C}_{v_*}(r) 
\quad\text{для всех  $0<r<R<+\infty$}.
\end{equation*}
Сложение этих равенств даёт  равенство
\begin{equation}\label{TC}
{\boldsymbol T}_U(r,R)\overset{\eqref{{TTN}T}}{=}
{\sf C}_{\sup\{u_*,v_*\}}(R)-{\sf C}_{v_*}(r) \in  \RR^+,
\quad 0< r<R\in  \RR^+.
\end{equation} 
Функция $f\colon I\to \RR$ выпукла (соответственно вогнута) на открытом интервале $I\subset \RR$
относительно строго возрастающей непрерывной функции $k\colon I\to \RR$, если 
суперпозиция  $f\circ k^{-1}$ выпукла (соответственно вогнута) на образе $k(I)\subset \RR$.    

Зависящую от размерности пространства $\RR^{\tt d}$ строго  возрастающую непрерывную на $ \RR^+$ функцию 
\begin{equation}\label{kKd-2}
\Bbbk_{\tt d-2} \colon  t\underset{0<t\in \RR^+}{\longmapsto} \begin{cases}
\ln t  &\text{\it  при ${\tt d=2}$},\\
 -\dfrac{1}{t^{\tt d-2}} &\text{\it при ${\tt d>2}$,} 
\end{cases} 
\quad \Bbbk (0):=-\infty \in \overline \RR,
\end{equation}
часто будем записывать без нижнего индекса  ${\tt d-2}$ как просто $\Bbbk$, когда возможные значения размерности 
$2\leq {\tt d}\in \NN$ ясны из формулировок  или контекста. 

Средние  по сфере  \eqref{CU} для субгармонических функций --- возрастающие и выпуклые относительно 
$\Bbbk_{\tt d-2}$  \cite[теорема 2.6.8]{Rans}, \cite[3.9]{HK}. Таким образом, из представления  \eqref{TC}
для разностной характеристики Неванлинны сразу следует 
\begin{propos}\label{proT} Разностная характеристика Неванлинны ${\boldsymbol T}_U$ $\delta$-суб\-г\-а\-р\-м\-о\-н\-и\-ч\-е\-с\-к\-ой функции $U\not\equiv \pm \infty$ на шаре с центром в нуле положительна,
 возрастающая и выпуклая относительно $\Bbbk_{\tt d-2}$ по второй большей переменной, а также 
убывающая  и вогнутая относительно $\Bbbk_{\tt d-2}$ по  первой переменной. 
\end{propos}

\section{Модуль непрерывности  меры и  основная теорема}

\begin{definition}\label{defmc}
 {\it Модуль непрерывности\/} меры  Бореля 
 $\mu$ на $\RR^{\tt d}$ ---  функция 
\begin{equation}\label{hmuR}
{\sf h}_{\mu}\colon t\underset{t\in \RR^+}{\longmapsto} \sup\limits_{y\in \RR^{\tt d}}\mu \bigl(\overline B_y(t)\bigr)
\overset{\eqref{muyr} }{=}\sup\limits_{y\in \RR^{\tt d}}\mu_y^{\rad}(t)\in \overline \RR^+.  
\end{equation}
\end{definition}

\begin{propos}\label{proh}
Пусть   $0<r\in \RR^+$, а  $\mu$ ---  мера Бореля на замкнутом шаре $\overline B(r)\subset \RR^{\tt d}$ полной меры   
\begin{equation}\label{Mmu}
{\tt M}:=\mu \bigl(\overline B(r)\bigr)=\mu (\RR^{\tt d})\in \overline \RR^+.
\end{equation}  
 Тогда ${\sf h}_{\mu}$  --- возрастающая функция и 
\begin{subequations}\label{hR}
\begin{align}
{\sf h}_{\mu}(t)&\leq {\tt M}\quad\text{при всех $t\in  \RR^+$}, 
\tag{\ref{hR}$\leq$}\label{hRMl}
\\ 
{\sf h}_{\mu}(t)&\equiv {\tt M}\quad\text{при всех $t\geq r$}.
\tag{\ref{hR}$\equiv$}\label{h<M}
\end{align}
\end{subequations}
\end{propos}
\begin{proof}
Из  \eqref{Mmu} по определению \eqref{hmuR}  возрастание ${\sf h}_{\mu}$ и неравенство \eqref{hRMl} очевидны. 
Согласно включению $\overline B(r)\subset \overline B(t)$  при $t\geq r$ имеем
$${\sf h}_{\mu}(t)\geq {\sf h}_{\mu}(r)\overset{\eqref{hmuR}}{=}
\sup\limits_{y\in \RR^{\tt d}}\mu \bigl(\overline B_y(t)\bigr)\geq
 \mu\bigl(\overline B(r)\bigr)\overset{\eqref{Mmu}}{=}{\tt M}
\quad\text{при всех  $t\geq r$},
$$ 
что вместе с \eqref{hRMl} даёт  \eqref{h<M}. 
\end{proof}

По теории мер и интегрирования  придерживаемся терминологии, но не обозначений,  
из монографий 
\cite{Federer}, 
\cite{EG}, 
\cite[Введение, \S~1]{Landkof}.
К примеру, как и в \cite{EG}, если интеграл от функции  по мере $\mu$ 
существует и принимает значение из $\overline \RR$, то эту функцию называем {\it $\mu$-интегрируемой,} а если этот интеграл ещё и конечен, т.е. со значением в $\RR$, то эту функцию  называем  {\it $\mu$-суммируемой.} 

\begin{maintheorem}\label{th3}
Пусть $0<r<R\in \RR^+$, а  $\mu$ ---  мера Бореля на замкнутом шаре $\overline B(r)\subset \RR^{\tt d}$ конечной полной меры  ${\tt M}<+\infty$ из \eqref{Mmu} и модулем непрерывности ${\sf h}_{\mu}$ из \eqref{hmuR}. 
Если выполнено условие
\begin{equation}\label{k0CR}
\int_0 \frac{{\sf h}_{\mu}(t)}{t^{\tt d-1}}\dd t< +\infty,
\end{equation} 
 то  $\mu$-суммируема любая  $\delta$-суб\-г\-а\-р\-м\-о\-н\-и\-ч\-е\-с\-к\-ая   функция   $U\not\equiv \pm \infty$ на $\overline B(R)$ и
\begin{subequations}\label{UR}
\begin{align}
\int_{\overline B(r)} U^+\dd \mu &\leq  
 A_{\tt d}(r,R){\boldsymbol  T}_U( r, R)
 \biggl({\tt M}+\int_0^{R+r}\frac{{\sf h}_{\mu}(t)}{t^{{\tt d}-1}}\dd t\biggr),
\tag{\ref{UR}T}\label{{UR}T}\\
\intertext{где $r$ в ${\boldsymbol  T}_U( r, R)$ можно заменить на любое  число  $r_0\in [0,r]$, а}
 A_{\tt d}(r,R)&:=2\Bigl(\frac{R+r}{R-r}\Bigr)^{\tt d-1}\max\Bigl\{1, (R-r)^{\tt d-2}\Bigr\}.
\tag{\ref{UR}A}\label{{UR}A}
\end{align}
\end{subequations}

\end{maintheorem}
\begin{remark} В силу \eqref{hRMl} и конечности полной меры ${\tt M}$ {\tiny }  условие \eqref{k0CR}  в рамках основной теоремы \eqref{k0CR}  можно переписать в эквивалентной  форме
 \begin{equation}\label{k0CRR}
\int_0^{R+r} \frac{{\sf h}_{\mu}(t)}{t^{\tt d-1}}\dd t
\overset{\eqref{kKd-2}}{=}\widehat{\tt d}\int_0^{R+r} {\sf h}_{\mu}(t)\dd \Bbbk_{\tt d-2}(t)< +\infty,
\end{equation} 
что  обеспечивает существование пределов 
\cite[предложение 2.2]{KhaShm19}
\begin{equation}\label{es:00lp0} 
\lim\limits_{0<t\to 0} {\sf h}_{\mu}(t)=0,\quad 
\quad \lim\limits_{0<t\to 0}{\sf h}_{\mu}(t)\Bbbk_{\tt d-2}( t)=0. 
\end{equation} 
\end{remark}

\begin{remark} 
Условие  \eqref{k0CR} означает, что функция $t\mapsto {\sf h}_{\mu}(t)t^{\tt 2-d}$ удовлетворяет классическому условию Дини. В частности, при ${\tt d}=2$, т.е. на $\CC$, условие \eqref{k0CR} --- это  уже в точности {\it условие Дини\/} в нуле
\begin{equation}\label{k0CR2}
\int_0 \frac{{\sf h}_{\mu}(t)}{t}\dd t< +\infty,
\end{equation} 
а заключительное неравенство \eqref{UR} упрощается до 
\begin{equation}\label{UR2}
\int_{\overline B(r)} U^+\dd \mu \leq  2\frac{R+r}{R-r} {\boldsymbol  T}_U( r, R)
\biggl({\tt M}+ \int_0^{R+r}\frac{{\sf h}_{\mu}(t)}{t}\dd t\biggr),
\end{equation}
где $r$ в ${\boldsymbol  T}_U( r, R)$  можно заменить на любое  $r_0\in [0,r]$.
Кроме того, для мероморфной функции $f\neq 0,\infty$ неравенство \eqref{UR2} для $\delta$-суб\-г\-а\-р\-м\-о\-н\-и\-ч\-е\-с\-к\-ой функции $U:=\ln |f|\not\equiv \pm\infty$ согласно \eqref{csTTNN} можно записать в традиционных обозначениях \eqref{{TN}T}, \eqref{{TN}N} в виде 
\begin{equation}\label{UR2f}
\int_{\overline B(r)} \ln^+|f|\dd \mu 
\\
\leq  2\frac{R+r}{R-r} \Bigl(T(R, f)-N(r, f)\Bigr)
\biggl({\tt M}+ \int_0^{R+r}\frac{{\sf h}_{\mu}(t)}{t}\dd t\biggr),
\end{equation}
где $r$ в $N( r,f)$ из правой части  можно заменить на любое  $r_0\in [0,r]$ с учётом \eqref{{actR}0} и $\ln 0\overset{\eqref{kKd-2}}{=}-\infty$. В частности, 
если  $r\geq 1$ или при отсутствии у $f$ полюса в нуле, т.е. при  $f(0)\in \CC$, имеем $N(r, f)\geq 0$ и 
\begin{equation}\label{UR2fr}
\int_{\overline B(r)} \ln^+|f|\dd \mu \overset{\eqref{UR2f}}{\leq} 2\frac{R+r}{R-r} T(R, f)
\biggl({\tt M}+ \int_0^{R+r}\frac{{\sf h}_{\mu}(t)}{t}\dd t\biggr).
\end{equation}
\end{remark}

\begin{proof}[основной теоремы]
При $\mu= 0$ всё очевидно, поэтому далее ${\tt M}\overset{\eqref{Mmu}}{>}0$. 
Пусть $u\not\equiv -\infty$ и $v\not\equiv -\infty$  --- пара субгармонических функций на $\overline B(R)$, представляющих $U=u-v$. Применяя формулу Пуассона\,--\,Йенсена  для шара  $B(R)$ \cite[(3.7.3)]{HK} для всех значений  $u(x)\neq -\infty$ и $v(x)\neq -\infty$ и вычитая одно равенство из другого, получаем 
\begin{multline}\label{Ux}
U(x)=\frac{1}{{\sf s}_{\tt d-1}}\int_{\partial \overline B(R)} \frac{R^2-|x|^2}{R|y-x|^{\tt d}}U(y)\dd \sigma_{\tt d-1}^R (y)\\
+\int_{B(R)}\Biggl( \Bbbk\biggl(\Bigl|\frac{R}{|y|}y-\frac{|y|}{R}x\Bigr| \biggr)-\Bbbk\bigl(|y-x|\bigr)\Biggr)\dd \varDelta_U(y)
\end{multline}
при всех $x\in \overline B(r)\setminus E$ вне борелевского {\it полярного множества\/} 
\begin{equation}\label{E}
E=\bigl\{x\in B(R)\bigm| u(x)=-\infty\bigr\}\bigcup \bigl\{x\in B(R)\bigm| v(x)=-\infty\bigr\}.
\end{equation} 
Для {\it положительного ядра Пуассона\/} \cite[1.5.4]{HK} имеем 
\begin{equation*}
\frac{1}{{\sf s}_{\tt d-1}} \frac{R^2-|x|^2}{R|y-x|^{\tt d}}\leq 
\frac{1}{{\sf s}_{\tt d-1}} \frac{R+r}{R(R-r)^{\tt d-1}}
\quad\text{при $y\in \partial \overline B(R)$ и $x\in \overline B(r)$},
\end{equation*}
а для {\it положительной функции Грина\/}  \cite[теорема 1.10]{HK} ---
\begin{equation*}
\Bbbk\biggl(\Bigl|\frac{R}{|y|}y-\frac{|y|}{R}x\Bigr| \biggr)-\Bbbk\bigl(|y-x|\bigr)\leq
\Bbbk(R+r)-\Bbbk\bigl(|y-x|\bigr)
\quad\text{при $y\in B(R)$ и $x\in \overline B(r)$}.
\end{equation*}
Таким образом, из \eqref{Ux} следует 
\begin{equation}\label{U+B}
U^+(x)\leq \frac{R^{\tt d-2}(R+r)}{(R-r)^{\tt d-1}}{\sf C}_{U^+}(R)
+\int_{B(R)}\Bigl(\Bbbk(R+r)-\Bbbk\bigl(|y-x|\bigr)\Bigr) \dd\varDelta_U^-(y)
\end{equation}
для всех $x\overset{\eqref{E}}{\in} \overline B(r)\setminus E$.
На $\overline B(r)\setminus E$ функция $U^+$ всюду определена как положительная часть разности  $u-v$
 полунепрерывных сверху функций $u$ и $v$ со значениями в $\RR$ и является измеримой по сужению меры $\mu$ на 
 $\overline B(r)\setminus E$, а  $\mu$-интегрируемость функции $U^+$ на всём шаре
$\overline B(r)$ обеспечивает 
\begin{lemma}\label{lem1}
Для меры Бореля  $\mu$ на $\overline B(r)$ конечной полной меры \eqref{Mmu} при условии 
\eqref{k0CR} $\mu$-мера борелевского полярного подмножества в $\overline B(r)$
равна нулю.
\end{lemma}
\begin{proof}[леммы \ref{lem1}] Условие \eqref{k0CR}, т.е. \eqref{k0CRR} вместе с \eqref{es:00lp0}, по  теореме Фростмана \cite[теорема II.1]{Carleson}, \cite[теорема 5.1.12]{HedbergAdams} означает, что 
${\sf h}_{\mu}$-мера Хаусдорфа  мажорирует меру $\mu$, и, согласно  \cite[теорема 5.13]{HK}, \cite[IV, теорема 1]{Carleson}, любое ограниченное борелевское полярное множество в $\RR^{\tt d}$, т.е. множество  нулевой ёмкости, 
 имеет нулевую  ${\sf h}_{\mu}$-меру Хаусдорфа. Отсюда и $\mu$-мера этого полярного множества равна нулю.
\end{proof} 
Теперь мы вправе интегрировать  по мере $\mu$ неравенство \eqref{U+B} и использовать для  правой части теорему Фубини  о повторных интегралах: 
\begin{multline*}
\int_{\overline B(r)}U^+\dd \mu
\leq \int_{\overline B(r)}U^+\frac{R^{\tt d-2}(R+r)}{(R-r)^{\tt d-1}}{\sf C}_{U^+}(R)\dd \mu(x)
\\+\int_{\overline B(r)}\int_{B(R)}\Bigl(\Bbbk(R+r)-\Bbbk\bigl(|y-x|\bigr)\Bigr) \dd\varDelta_U^-(y)\dd \mu (x)\\
=\frac{R^{\tt d-2}(R+r)}{(R-r)^{\tt d-1}}{\sf C}_{U^+}(R)\mu\bigl(\overline B(r)\bigr)
\\+\int_{B(R)}\int_{\overline B(r)}\Bigl(\Bbbk(R+r)-\Bbbk\bigl(|y-x|\bigr)\Bigr) \dd \mu (x)\dd\varDelta_U^-(y)
\\
\overset{\eqref{Mmu}}{=}
\frac{R^{\tt d-2}(R+r)}{(R-r)^{\tt d-1}}{\sf C}_{U^+}(R)\,{\tt M}
+\int_{B(R)}\int_0^r\Bigl(\Bbbk(R+r)-\Bbbk(t)\Bigr) \dd \mu_y^{\rad} (t)\dd\varDelta_U^-(y).
\end{multline*}
Здесь для последнего внутреннего интеграла Римана\,--\,Стилтьеса интегрированием по частям получаем 
\begin{multline*}
\int_0^r\Bigl(\Bbbk(R+r)-\Bbbk(t)\Bigr) \dd \mu_y^{\rad} (t)
\leq \int_0^{R+r}\Bigl(\Bbbk(R+r)-\Bbbk(t)\Bigr) \dd \mu_y^{\rad} (t)\\
\leq \int_0^{R+r} \mu_y^{\rad} (t)\dd \Bbbk(t)
\overset{\eqref{hmuR}}{\leq}
\int_0^{R+r}{\sf h}_{\mu}(t)\dd \Bbbk(t)
\overset{\eqref{es:00lp0}}{=}
\widehat{\tt d}\int_0^{R+r}\frac{{\sf h}_{\mu}(t)}{t^{\tt d-1}}\dd t.
\end{multline*}
Таким образом, из последних двух цепочек (не)равенств имеем  
\begin{multline*}
\int_{\overline B(r)}U^+\dd \mu
\leq \frac{R^{\tt d-2}(R+r)}{(R-r)^{\tt d-1}}{\sf C}_{U^+}(R)\,{\tt M}
+
\int_{B(R)}\widehat{\tt d}\int_0^{R+r}\frac{{\sf h}_{\mu}(t)}{t^{\tt d-1}}\dd t
\dd\varDelta_U^-(y)\\
=\frac{R^{\tt d-2}(R+r)}{(R-r)^{\tt d-1}}{\sf C}_{U^+}(R)\,{\tt M}
+(\varDelta_U^-)^{\rad}(R)\widehat{\tt d}\int_0^{R+r}\frac{{\sf h}_{\mu}(t)}{t^{\tt d-1}}\dd t. 
\end{multline*}
Рассуждения и выкладки проводились для $0<r<R\in \RR^+$, поэтому последнее  
верно и для любого $R_*\in (r,R)$, что позволяет переписать его как
\begin{equation}\label{R*}
\int_{\overline B(r)}U^+\dd \mu
\leq \frac{R_*^{\tt d-2}(R_*+r)}{(R_*-r)^{\tt d-1}}{\sf C}_{U^+}(R_*)\,{\tt M}
+(\varDelta_U^-)^{\rad}(R_*)\widehat{\tt d}\int_0^{R_*+r}\frac{{\sf h}_{\mu}(t)}{t^{\tt d-1}}\dd t. 
\end{equation}

\begin{lemma}\label{lemd} Пусть 
$\varDelta$ --- конечная мера  на $\overline B(R)$ и $0<R_*<R$.
 Тогда 
\begin{equation}\label{dBr}
\varDelta\bigl(\overline B(R_*)\bigr)\overset{\eqref{muyr}}{=:}
\varDelta^{\rad}(R_*)
\leq 
\frac{R^{\tt d-1}}{{\widehat{\tt d}}(R-R_*)} 
{\sf N}_{\varDelta} (R_*,R).
\end{equation}
\end{lemma} 
\begin{proof}[леммы \ref{lemd}] В силу возрастания считающей функции $\varDelta^{\rad}$
\begin{equation*}
\varDelta^{\rad}(R_*)\leq \int_{R_*}^{R}\frac{\varDelta^{\rad}(t)}{t^{\tt d-1}}\dd t\biggm/
\int_{R_*}^{R} \frac{\dd t}{t^{\tt d-1}}
\overset{\eqref{Ntt}}{\leq} \frac{1}{{\widehat{\tt d}}}{\sf N}_{\varDelta} (r_*,R)\biggm/ 
 \frac{1}{R^{\tt d-1}}\int_{R_*}^R\dd t.
\end{equation*}
\end{proof}
С помощью  леммы \ref{lemd} можем продолжить неравенство \eqref{R*} как 
\begin{multline*}
\int_{\overline B(r)}U^+\dd \mu
\leq \frac{R_*^{\tt d-2}(R_*+r)}{(R_*-r)^{\tt d-1}}{\sf C}_{U^+}(R_*)\,{\tt M}
+\frac{R^{\tt d-1}}{R-R_*} {\sf N}_{\varDelta} (R_*,R)\int_0^{R_*+r}\frac{{\sf h}_{\mu}(t)}{t^{\tt d-1}}\dd t\\
 \leq \frac{R_*^{\tt d-2}(R_*+r)}{(R_*-r)^{\tt d-1}}{\sf C}_{U^+}(R_*)\,{\tt M}
+\frac{R^{\tt d-1}}{R-R_*} {\sf N}_{\varDelta} (r,R)\int_0^{R+r}\frac{{\sf h}_{\mu}(t)}{t^{\tt d-1}}\dd t\\
\leq\max\biggl\{\frac{R_*^{\tt d-2}(R_*+r)}{(R_*-r)^{\tt d-1}}, \frac{R^{\tt d-1}}{R-R_*}\biggr\}
\cdot \max\Bigl\{{\sf C}_{U^+}(R_*), {\sf N}_{\varDelta} (r,R)\Bigr\}
\biggl({\tt M}+\int_0^{R+r}\frac{{\sf h}_{\mu}(t)}{t^{\tt d-1}}\dd t\biggr),
\end{multline*}
где по определению \ref{defT}  из возрастания по предложению \ref{proT} разностной характеристики Неванлинны по второму аргументу 
для последнего максимума имеем 
\begin{equation*}
\max\Bigl\{{\sf C}_{U^+}(R_*), {\sf N}_{\varDelta} (r,R)\Bigr\}\overset{\eqref{{TTN}T}}{\leq} 
\max\Bigl\{{\boldsymbol T}_{U}(r,R_*), {\boldsymbol T}_{U}(r,R)\Bigr\}= {\boldsymbol T}_{U}(r,R), 
\end{equation*}
а при выборе $R_*:=\frac{R+r}{2}$ для первого максимума в правой части 
\begin{multline*}
\max\biggl\{\frac{R_*^{\tt d-2}(R_*+r)}{(R_*-r)^{\tt d-1}}, \frac{R^{\tt d-1}}{R-R_*}\biggr\}
=\max \biggl\{\frac{(R+r)^{\tt d-2}(R+2r)}{(R-r)^{\tt d-1}}, \frac{2R^{\tt d-1}}{R-r}\biggr\}\\
\leq 2\Bigl(\frac{R+r}{R-r}\Bigr)^{\tt d-1}\max\Bigl\{1, (R-r)^{\tt d-2}\Bigr\}
\overset{\eqref{{UR}A}}{=} A_{\tt d}(r,R),
\end{multline*}
что и доказывает неравенство  \eqref{{UR}T} вместе с конечностью правой части в нём в силу условия 
\eqref{k0CR} в форме \eqref{k0CRR}. Конечность правой части   \eqref{{UR}T}  означает и $\mu$-суммируемость функции $U^+$.
\end{proof}

\end{fulltext}

\end{document}